\def\benm{\begin{enumerate}}
\def\eenm{\end{enumerate}}

\def\bal{\begin{align}}
\def\eal{\end{align}}

\def\st{\stackrel{\tau}{\ast}}
\def\lst{\stackrel{\tau_l}{\ast}}
\def\rst{\stackrel{\tau_r}{\ast}}

\documentclass{amsart}
\usepackage{amsmath, amsthm, amscd, amsfonts,amssymb, graphicx}

\newtheorem{theorem}{Theorem}[section]

\theoremstyle{definition}

\newtheorem{proposition}[theorem]{Proposition}
\newtheorem{corollary}[theorem]{Corollary}

\theoremstyle{remark}
\newtheorem{remark}[theorem]{Remark}

\numberwithin{equation}{section}

\begin{document}

\title[A new approach to convolution and semi-direct products of groups]{A new approach to convolution and semi-direct products of groups}

\author[Arash Ghaani Farashahi]{Arash Ghaani Farashahi$^{*}$}
\address{$^1$ Department of Pure Mathematics, Faculty of Mathematical sciences, Ferdowsi University of
Mashhad (FUM), P. O. Box 1159, Mashhad 91775, Iran.}

\email{ghaanifarashahi.arash@stu-mail.um.ac.ir}
\email{ghaanifarashahi@hotmail.com}
\email{ghaanifarashahi@gmail.com}

\curraddr{}


\author{Rajabali Kamyabi-Gol}
\address{$^2$ Department of Pure Mathematics, Faculty of Mathematical sciences, Ferdowsi University of
Mashhad (FUM), P. O. Box 1159, Mashhad 91775, Iran.}
\address{Center of Excellence in Analysis on Algebraic Structures (CEAAS),
Ferdowsi University of Mashhad (FUM), P. O. Box 1159, Mashhad 91775, Iran.}
\email{kamyabi@ferdowsi.um.ac.ir}


\subjclass[2000]{Primary 44A35, 22A10}

\date{}


\keywords{semi-direct products of groups, left $\tau$-convolution ($\tau_l$-convolution), right $\tau$-convolution ($\tau_r$-convolution), $\tau$-convolution, $\tau$-involution, $\tau$-approximate identity, closed $\tau$-ideal.}
\thanks{$^*$Corresponding author}
\thanks{E-mail addresses: ghaanifarashahi@hotmail.com (A. Ghaani Farashahi). kamyabi@ferdowsi.um.ac.ir (R. Kamyabi-Gol)
}

\begin{abstract}
Let $H$ and $K$ be locally compact groups and $\tau:H\to Aut(K)$ be a continuous homomorphism and also let $G_\tau=H\ltimes_\tau K$ be the semi-direct product of $H$ and $K$ with respect to $\tau$. We define left and also right $\tau$-convolution on $L^1(G_\tau)$ such that $L^1(G_\tau)$ with respect to each of them is a Banach algebra. Also we define $\tau$-convolution as a linear combination of the left and right $\tau$-convolution. We show that the $\tau$-convolution is commutative if and only if $K$ is abelian and also when $H$ and $K$ are second countable groups, the $\tau$-convolution coincides with the standard convolution of $L^1(G_\tau)$ if and only if $H$ is the trivial group. We prove that there is a $\tau$-involution on $L^1(G_\tau)$ such that $L^1(G_\tau)$ with respect to the $\tau$-involution and $\tau$-convolution is a non-associative Banach $*$-algebra and also it is also shown that when $K$ is abelian, the $\tau$-involution and $\tau$-convolution makes $L^1(G_\tau)$ into a Jordan Banach $*$-algebra.
\end{abstract}

\maketitle

\section{\bf{Introduction}}

In mathematical analysis and in particular from functional analysis viewpoint, convolution on Euclidean space is a linear map on the function spaces related to Euclidean spaces, for more details see \cite{FollR}, also convolution is similar to cross-correlation.
In classical harmonic analysis convolution is an integral which interpret the quantity of overlap of a function as it is shifted over another function, see \cite{FollH} or \cite{HR1}. It has applications that include statistics methods, image and signal processing, electrical engineering, and also differential equations.
The convolution can be defined for functions on groups other than Euclidean space, see \cite{FollH} or \cite{HR1}. In particular, the circular convolution can be defined for periodic functions (that is, functions on the circle group), and the discrete convolution can be defined for functions on the set of integers. On the other hand one of the most important concepts in Fourier theory, Gabor theory, Wavelet theory, Multiresolution analysis and also in crystallography, is that of a convolution.

Many non-abelian groups in harmonic analysis can be considered as a semi-direct products of groups.
In this paper we want to introduce a new approach to the convolution on the semi-direct products of groups and also we study the basic properties of this theory.
The principal role played by convolutions in classical harmonic analysis is in evidence throughout \cite{BH},\cite{JEW},\cite{50},\cite{Zyg}.
Since convolution plays an important role in general theory of harmonic analysis, we focuss on the convolution theory and we define a convolution on the functions spaces related to semidirect products of groups.
More precisely many classes of non-abelian groups can be written as a semidirect product of groups in which one of the groups is abelian. In this case the semi-direct product group is non-abelian and so the standard convolution is noncommutative and so there is a lack to study the harmonic analysis on semi-direct products of groups. Thus it maybe worthwhile if we define a new convolution such that it's commutativity behavior depends only on one of the groups. In fact we are looking for a convolution on these classes of groups in which allow us to studying and analyzing the convolution with respect to the second group in the semi-direct product.

Throughout this article which contains 4 section we assume that $H$ and $K$ are locally compact topological groups and $\tau:H\to Aut(K)$ is a continuous homomorphism and $G_\tau=H\ltimes_\tau K$ is the semi-direct product of $H$ and $K$ with respect to $\tau$. Section 2 is devoted to fix notations and also a brief summary on the classic properties of the semi-direct products of groups. In section 3 first we define left and right $\tau$-convolution on $L^1(G_\tau)$ which makes $L^1(G_\tau)$ into a Banach algebra. Then we define the $\tau$-convolution as a linear combination of left and right convolution and also we define a $\tau$-involution on $L^1(G_\tau)$ and we show that $L^1(G_\tau)$ is a non-associative Banach $*$-algebra with respect to the $\tau$-convolution and the $\tau$-involution. We prove that the $\tau$-convolution is commutative if and only if $K$ is abelian.  We also show that this new algebra structure is completely different from the standard algebra on the $L^1$-function algebra related to each locally compact group.
Recall that our main idea on this theory is that when $H$ is the trivial group we get the classical convolution on the locally compact group $K$. Also we prove that the $\tau$-convolution coincides with the standard convolution of $L^1(G_\tau)$ if and only if $H$ is the trivial group.
Finally, in section 4 as an application for $p>1$ we make $L^p(G_\tau)$ into a left Banach $L^1_{\tau_l}(G_\tau)$-module.

\section{\bf{Preliminaries and notations}}

A non-associative algebra $\mathcal{B}$ is a linear space $\mathcal{B}$ over field of complex or real numbers endowed with a bilinear map $(x,y)\mapsto xy$ from $\mathcal{B}\times\mathcal{B}$ into $\mathcal{B}$ and also a Jordan algebra is a commutative non-associative algebra $\mathcal{B}$ whose product satisfies the Jordan identity $(xy)x^2=x(yx^2)$ for all $x,y\in \mathcal{B}$.
Note that the term non-associative stands for not necessarily associative. More precisely we use the term non-associative algebra in order to emphasize that the associativity of the product is not being assumed.
A non-associative Banach algebra is a non-associative algebra $\mathcal{B}$ over the field of complex or real numbers, whose underlying linear space is a Banach space with respect to a norm $\|.\|$ satisfying $\|xy\|\le\|x\|\|y\|$ for all $x,y\in \mathcal{B}$.

Let $X$ be locally compact Hausdorff space. By $\mathcal{C}_c(X)$ we mean the space of all continuous complex valued functions on $X$ with compact supports. If $\mu$ is a positive Radon measure on $X$, for each $1\le p<\infty$ the Banach space of equivalence classes of $\mu$-measurable complex valued functions $f:X\to\mathbb{C}$ such that
$$\|f\|_p=\left(\int_X|f(x)|^pd\mu(x)\right)^{1/p}<\infty$$
is denoted by $L^p(X,\mu)$ which contains dense subspace $\mathcal{C}_c(X)$.
When $G$ is a locally compact group with left Haar measure $dx$ and modular function $\Delta_G$
for each $p\ge 1$ we mean by $L^p(G)$ the Banach space $L^p(G,dx)$. If $p=1$ standard convolution of $f,g\in L^1(G)$ defined via
\begin{equation}\label{0.0}
f\ast g(x)=\int_Gf(y)g(y^{-1}x)dy
\end{equation}
and also the standard involution of $f\in L^1(G)$ is given by $f^*(x)=\Delta_G(x^{-1})\overline{f(x)}$
which makes $L^1(G)$ into a Banach $*$-algebra. We recall that the Banach $*$-algebra $L^1(G)$ with respect to the standard convolution given in (\ref{0.0}) has an approximate identity (see Proposition 2.42 of \cite{FollH}) and also the Banach $*$-algebra $L^1(G)$ is commutative if and only if $G$ is abelian.

Let $H$ and $K$ be locally compact groups with identity elements $e_H$ and $e_K$ respectively and left Haar measures $dh$ and $dk$ respectively and also let $\tau:H\to Aut(K)$ be a homomorphism such that the map $(h,k)\mapsto \tau_h(k)$ is continuous from $H\times K$ to $K$. In this case we say the homomorphism $\tau:H\to Aut(K)$ is continuous.
The semidirect product $G_\tau=H\ltimes_\tau K$ is a locally compact topological group with underlying set $H\times K$ which is equipped with product topology and group operation is defined by
\begin{equation}\label{}
(h,k)\ltimes_\tau(h',k')=(hh',k\tau_h(k'))\hspace{0.5cm}{\rm and}\hspace{0.5cm}(h,k)^{-1}=(h^{-1},\tau_{h^{-1}}(k^{-1})).
\end{equation}
The left Haar measure of $G_\tau$ is $d\mu_{G_\tau}(h,k)=\delta(h)dhdk$ and the modular function of $G_\tau$ is $\Delta_{G_\tau}(h,k)=\delta(h)\Delta_H(h)\Delta_K(k)$, where the positive continuous homomorphism $\delta:H\to(0,\infty)$ is given by
$dk=\delta(h)d(\tau_h(k))$ and also $\Delta_H$ and $\Delta_K$ are modular functions of locally compact groups $H$ and $K$ respectively, for more details see Theorem 15.29 of \cite{HR1}.

\section{\bf{$\tau$-convolution and $\tau$-involution}}

In the following we define a $\tau$-convolution and a $\tau$-involution on $L^1(G_\tau)$ which is different from the usual convolution and involution of $L^1(G_\tau)$. We note that our idea on this extension is that only when $H$ is the trivial group this algebra structure coincides with the standard algebra structure of $L^1(K)$.
For $f\in L^1(G_\tau)$, let $\widetilde{f}\in L^1(K)$ be given by
\begin{equation}\label{9}
\widetilde{f}(k):=\int_Hf_{t}(k)\delta(t)dt,
\end{equation}
where for each $t\in H$ the function $f_t$ is defined for a.e. $k$ in $K$ via $f_t(k)=f(t,k)$. Then the integral defined in (\ref{9}) converges. In fact we have
\begin{align*}
\|\widetilde{f}\|_{L^1(K)}
&=\int_K|\widetilde{f}(k)|dk
\\&\le\int_K\left(\int_H|f(t,k)|\delta(t)dt\right)dk=\|f\|_{L^1(G_\tau)}.
\end{align*}
For $f,g\in L^1(G_\tau)$, we define the right $\tau$-convolution on $L^1(G_\tau)$ by
\begin{equation}\label{10.r}
f\rst g(h,k):=\int_Hf_{h}\ast g_{t}(k)\delta(t)dt.
\end{equation}
where $f_h\ast g_t$ is the standard convolution on $L^1(K)$. We recall that according to the Fubini-Toneli Theorem, for each $f\in L^1(G_\tau)$ we have $f_h\in L^1(K)$ for a.e. $h$ in $H$. The integral defined in (\ref{10.r}) converges and also for each $f,g\in L^1(G_\tau)$ we have $\|f\rst g\|_{L^1(G_\tau)}\le\|f\|_{L^1(G_\tau)}\|g\|_{L^1(G_\tau)}$. Indeed by using Fubini's Theorem and Proposition 2.39 of \cite{FollH} we get
\begin{align*}
\|f\rst g\|_{L^1(G_\tau)}
&=\int_H\int_K\left|\int_Hf_{h}\ast g_{t}(k)\delta(t)dt\right|\delta(h)dkdh
\\&\le\int_H\int_K\int_H|f_{h}\ast g_{t}(k)|\delta(h)\delta(t)dtdkdh
\\&\le\int_H\int_H\left(\int_K|f_{h}\ast g_{t}(k)|dk\right)\delta(h)\delta(t)dtdh
\\&=\int_H\int_H\|f_h\ast g_t\|_{L^1(K)}\delta(h)\delta(t)dtdh
\\&\le\int_H\int_H\|f_h\|_{L^1(K)}\|g_t\|_{L^1(K)}\delta(h)\delta(t)dtdh=\|f\|_{L^1(G_\tau)}\|g\|_{L^1(G_\tau)}.
\end{align*}
The right $\tau$-convolution for a.e $(h,k)\in G_\tau$ satisfies $f\rst g(h,k)=f_h\ast\widetilde{g}(k)$.
This is because for a.e. $(h,k)\in G_\tau$ we have
\begin{align*}
f\rst g(h,k)
&=\int_Hf_h\ast g_t(k)\delta(t)dt
\\&=\int_H\left(\int_Kf(h,s)g(t,s^{-1}k)ds\right)\delta(t)dt
\\&=\int_Kf(h,s)\left(\int_Hg(t,s^{-1}k)\delta(t)dt\right)ds=f_h\ast\widetilde{g}(k).
\end{align*}
Similarly the left $\tau$-convolution on $L^1(G_\tau)$ is defined by
\begin{equation}\label{10.l}
f\lst g(h,k):=\int_Hf_{t}\ast g_{h}(k)\delta(t)dt,
\end{equation}
which analogy for a.e. $(h,k)\in G_\tau$ satisfies $f\lst g(h,k)=\widetilde{f}\ast g_h(k)$ and also
$\|f\lst g\|_{L^1(G_\tau)}\le\|f\|_{L^1(G_\tau)}\|g\|_{L^1(G_\tau)}.$
Now the $\tau$-convolution of $f,g\in L^1(G_\tau)$ is defined as
\begin{equation}\label{11}
f\st g=2^{-1}\left(f\rst g+f\lst g\right).
\end{equation}
Then clearly $f\ast g\in L^1(G_\tau)$. More precisely we have
\begin{align*}\label{12}
\|f\st g\|_{L^1(G_\tau)}
&=2^{-1}\|f\rst g+f\lst g\|_{L^1(G_\tau)}
\\&\le2^{-1}\left(\|f\rst g\|_{L^1(G_\tau)}+\|f\lst g\|_{L^1(G_\tau)}\right)\le \|f\|_{L^1(G_\tau)}\|g\|_{L^1(G_\tau)}.
\end{align*}
Note that the $\tau$-convolution of $f,g\in L^1_\tau(G_\tau)$ defined in (\ref{11}), for a.e. $(h,k)\in G_\tau$ can be rewritten in the following form
\begin{equation}\label{13}
f\st g(h,k)=2^{-1}\left(f_h\ast\widetilde{g}(k)+\widetilde{f}\ast g_h(k)\right).
\end{equation}
\begin{theorem}\label{11.1}
Let $\tau:H\to Aut(K)$ be a continuous homomorphism and $G_\tau=H\ltimes_\tau K$.
The right $\tau$-convolution defined in (\ref{10.r}) makes $L^1(G_\tau)$ into a Banach algebra.
\end{theorem}
\begin{proof}
Let $f,g,u$ in $L^1(G_\tau)$. Using the associativity of the standard convolution on $L^1(K)$ for a.e. $(h,k)\in G_\tau$ we have
\begin{align*}
(f\rst g)\rst u(h,k)
&=(f\rst g)_h\ast\widetilde{u}(k)
\\&=\int_K(f\rst g)_h(s)\widetilde{u}(s^{-1}k)ds
\\&=\int_K\left(\int_Hf_h\ast g_t(s)\delta(t)dt\right)\widetilde{u}(s^{-1}k)ds
\\&=\int_H\left(\int_Kf_h\ast g_t(s)\widetilde{u}(s^{-1}k)ds\right)\delta(t)dt
\\&=\int_H(f_h\ast g_t)\ast\widetilde{u}(k)\delta(t)dt
\\&=\int_Hf_h\ast(g_t\ast\widetilde{u})(k)\delta(t)dt
\\&=\int_H\left(\int_Kf_h(s)(g_t\ast\widetilde{u})(s^{-1}k)ds\right)\delta(t)dt
\\&=\int_H\left(\int_Kf_h(s)(g\rst u)_t(s^{-1}k)ds\right)\delta(t)dt
\\&=\int_Hf_h\ast(g\rst u)_t(k)\delta(t)dt
=f\rst(g\rst u)(h,k).
\end{align*}
\end{proof}
Next theorem shows the same result as in Theorem \ref{11.1} for the left $\tau$-convolution.
\begin{theorem}\label{11.1.2}
Let $\tau:H\to Aut(K)$ be a continuous homomorphism and $G_\tau=H\ltimes_\tau K$.
The left $\tau$-convolution defined in (\ref{10.l}) makes $L^1(G_\tau)$ into a Banach algebra.
\end{theorem}
\begin{proof}
We show that the left $\tau$-convolution is associative. Let $f,g,u\in L^1(G_\tau)$. Then, for a.e. $(h,k)\in G_\tau$ we have
\begin{align*}
(f\lst g)\lst u(h,k)
&=\int_H(f\lst g)_t\ast u_h(k)\delta(t)dt
\\&=\int_H(\widetilde{f}\ast g_t)\ast u_h(k)\delta(t)dt
\\&=\int_H\widetilde{f}\ast(g_t\ast u_h)(k)\delta(t)dt
=f\lst(g\lst u)(h,k).
\end{align*}
\end{proof}
But the following proposition guarantees that the $\tau$-convolution is not associative in general.
\begin{proposition}\label{11.2}
{\it The $\tau$-convolution defined in (\ref{11}), for each $f,g,u\in L^1(G_\tau)$ satisfies
\begin{equation}\label{11.2.1}
(f\st g)\st u-f\st (g\st u)=f\rst g\rst u-f\lst g\lst u.
\end{equation}
}\end{proposition}
\begin{proof}
Let $f,g,u\in L^1(G_\tau)$. Using Theorem \ref{11.1} and also Theorem \ref{11.1.2}, left and right $\tau$-convolutions are associative.
Thus it can be easily checked that for a.e. $(h,k)\in G_\tau$ we have
\begin{equation}\label{11.2.a}
f\rst (g\lst u)(h,k)=f\rst(g\rst u)(h,k)=(f\rst g)\rst u(h,k),
\end{equation}
\begin{equation}\label{11.2.b}
(f\rst g)\lst u(h,k)=f\lst (g\lst u)(h,k)=(f\lst g)\lst u(h,k).
\end{equation}
Also, using definition of the $\tau$-convolution and also (\ref{11.2.a}) and (\ref{11.2.b}) we have
\begin{align*}
(f\st g)\st u
&=2^{-1}\left((f\st g)\rst u+(f\st g)\lst u\right)
\\&=2^{-2}\left((f\rst g+f\lst g)\rst u+(f\rst g+f\lst g)\lst u\right)
\\&=2^{-2}\left((f\rst g)\rst u+(f\lst g)\rst u+(f\rst g)\lst u+(f\rst g)\rst u\right),
\end{align*}
and also similarly
\begin{align*}
f\st (g\st u)
&=2^{-1}\left(f\rst (g\st u)+f\lst (g\st u)\right)
\\&=2^{-2}\left(f\rst (g\rst u+g\lst u)+f\lst (g\rst u+g\lst u)\right)
\\&=2^{-2}\left(f\rst (g\rst u)+f\rst (g\lst u)+f\lst (g\rst u)+f\lst (g\lst u)\right).
\end{align*}
Now a straight forward calculation implies (\ref{11.2.1}).
\end{proof}
For $f\in L^1(G_\tau)$ let the $\tau$-involution be denoted by $f^{*^\tau}$ and defined for a.e. $(h,k)\in G_\tau$ by
\begin{equation}\label{14}
f^{*^\tau}(h,k):=(f_h)^*(k),
\end{equation}
where $(f_h)^*$ is the standard involution of $f_h$ in $L^1(K)$ given by $$(f_h)^*(k)=\Delta_K(k^{-1})\overline{f_h(k^{-1})},$$
which satisfies $\|(f_h)^*\|_{L^1(K)}=\|f_h\|_{L^1(K)}$ and $(f_h)^{*^*}=f_h$. Then, we have $f^{*^\tau}\in L^1(G_\tau)$. More precisely the linear map $*^\tau:L^1(G_\tau)\to L^1(G_\tau)$ is an isometry, because
\begin{align*}
\|f^{*^\tau}\|_{L^1(G_\tau)}&=\int_H\int_K|f^{*^\tau}(h,k)|\delta(h)dkdh
\\&=\int_H\left(\int_K|f_h^*(k)|dk\right)\delta(h)dh
\\&=\int_H\|f_h\|_{L^1(K)}\delta(h)dh=\|f\|_{L^1(G_\tau)}.
\end{align*}
Thus we prove the following theorem.
\begin{theorem}\label{11.3}
Let $\tau:H\to Aut(K)$ be a continuous homomorphism and $G_\tau=H\ltimes_\tau K$. $L^1(G_\tau)$ is a non-associative Banach $*$-algebra with respect to the $\tau$-convolution defined in (\ref{11}) and also the $\tau$-involution defined in (\ref{14}).
\end{theorem}
\begin{proof}
It remains to show that the $\tau$-involution is an anti-homomorphism. Let $f,g\in L^1(G_\tau)$. Using the anti-homomorphism property of the involution on $L^1(K)$, for a.e. $(h,k)\in G_\tau$ we have
\begin{align*}
(f\st g)^{*_\tau}(h,k)
&=(f\rst g)^{*_\tau}(h,k)+(f\lst g)^{*_\tau}(h,k)
\\&=\{(f\rst g)_h\}^*(k)+\{(f\lst g)_h\}^*(k)
\\&=(f_h\ast\widetilde{g})^*(k)+(\widetilde{f}\ast g_h)^*(k)
\\&=(\widetilde{g})^*\ast f_h^*(k)+g_h^*\ast(\widetilde{f})^*(k)=g^{*_\tau}\st f^{*_\tau}(h,k).
\end{align*}
\end{proof}
From now on, the notation $L^1_\tau(G_\tau)$ stands for the non-associative Banach $*$-algebra mentioned in Theorem \ref{11.3} and also we use the notations $L^1_{\tau_l}(G_\tau)$ and $L^1_{\tau_r}(G_\tau)$ for the Banach algebra mentioned in Theorem \ref{11.1} respectively Theorem \ref{11.1.2}.
A left(right) $\tau_l$-ideal of $L^1_{\tau_l}(G_\tau)$ is a subspace $\mathcal{I}$ of $L^1(G_\tau)$ such that for each $f\in \mathcal{I}$ and $g\in L^1(G_\tau)$ we have $g\lst f\in \mathcal{I}$ ($f\lst g\in \mathcal{I}$) and by a closed left(right) $\tau_l$-ideal we mean a $\|.\|_{L^1(G_\tau)}$-closed left(right) $\tau_l$-ideal in $L^1_\tau(G_\tau)$. Note that the same definitions can be considered for left, right or two sided $\tau_r$-ideals of $L^1_{\tau_r}(G_\tau)$.
In the following, we introduce a tool which transfer elements of $L^1(K)$ into the elements of $L^1(G_\tau)$. In fact we define a $*$-homomorphism form $L^1(K)$ into $L^1_\tau(G_\tau)$.
For $\Phi$ in $L^1(G_\tau)$ and $\psi$ in $L^1(K)$, let
\begin{equation}\label{15}
\Phi(\psi)(h,k):=\psi(k)\int_K\Phi(h,s)ds.
\end{equation}
Then $\Phi(\psi)\in L^1_\tau(G_\tau)$ with $\|\Phi(\psi)\|_{L^1_\tau(G_\tau)}\le\|\Phi\|_{L^1_\tau(G_\tau)}\|\psi\|_{L^1(K)}$. Indeed we have
\begin{align*}
\int_H\int_K|\Phi(\psi)(h,k)|\delta(h)dkdh
&=\int_H\int_K\left|\psi(k)\int_K\Phi(h,s)ds\right|dk\delta(h)dh
\\&\le\int_H\int_K\int_K|\psi(k)||\Phi(h,s)|dsdk\delta(h)dh
\\&=\left(\int_K|\psi(k)|dk\right)\left(\int_H\int_K|\Phi(h,s)|ds\delta(h)dh\right)
=\|\Phi\|_{L^1_\tau(G_\tau)}\|\psi\|_{L^1(K)}.
\end{align*}
We consider $L_1^+$ as the set of all nonnegative $\Phi$ in $L^1(G_\tau)$ with $\|\Phi\|_{L^1(G_\tau)}=1$.
\begin{proposition}\label{19}
{\it Let $\tau:H\to Aut(K)$ be a continuous homomorphism and $G_\tau=H\ltimes_\tau K$. Then, for each $\Phi\in L^+_1$, $f\in L^1(G_\tau)$ and also $\psi\in L^1(K)$ we have
\begin{enumerate}
\item $\displaystyle\widetilde{f}\ast\psi(k)=\int_Hf_t\ast\psi(k)\delta(t)dt$\ for\ a.e. $k\in K$,
\item $\displaystyle\psi\ast\widetilde{f}(k)=\int_H\psi\ast f_t(k)\delta(t)dt$\ for \ a.e. $k\in K$,
\item $\Phi(\psi)\st f(h,k)=2^{-1}\left(\|\Phi_h\|_{L^1(K)}\psi\ast\widetilde{f}(k)+\psi\ast f_h(k)\right)$\ for \ a.e. $(h,k)\in G_\tau$,
\item $f\st\Phi(\psi)(h,k)=2^{-1}\left(f_h\ast\psi(k)+\|\Phi_h\|_{L^1(K)}\widetilde{f}\ast\psi(k)\right)$\ for \ a.e. $(h,k)\in G_\tau$
\end{enumerate}
}\end{proposition}
\begin{proof} (1) Let $f\in L^1(G_\tau)$ and $\psi\in L^1(K)$. Then for a.e. $k\in K$ we have
\begin{align*}
\widetilde{f}\ast\psi(k)&=\int_K\widetilde{f}(s)\psi(s^{-1}k)ds
\\&=\int_K\left(\int_Hf(t,s)\delta(t)dt\right)\psi(s^{-1}k)ds
\\&=\int_H\left(\int_Kf(t,s)\psi(s^{-1}k)ds\right)\delta(t)dt=\int_Hf_t\ast\psi(k)\delta(t)dt.
\end{align*}
Similarly for a.e $k\in K$ we have
$$\displaystyle\psi\ast\widetilde{f}(k)=\int_H\psi\ast f_t(k)\delta(t)dt.$$
(3) Let $\Phi,f\in L^1(G_\tau)$ and $\psi\in L^1(K)$. Using (\ref{13}) for a.e. $(h,k)\in G_\tau$ we have
\begin{align*}
\Phi(\psi)\st f(h,k)&=2^{-1}\left(\left(\Phi(\psi)\right)_h\ast\widetilde{f}+\widetilde{\Phi(\psi)}\ast f_h(k)\right)
\\&=2^{-1}\left(\psi\ast\widetilde{f}(k)\|\Phi_h\|_{L^1(K)}+\psi\ast f_h(k)\right).
\end{align*}
Similarly for a.e $(h,k)\in G_\tau$ we have
$$f\st\Phi(\psi)(h,k)=2^{-1}\left(\widetilde{f}\ast\psi(k)\|\Phi_h\|_{L^1(K)}+f_h\ast\psi(k)\right).$$
\end{proof}
Next theorem shows that each $\Phi\in L^+_1$ defines a $*$-homomorphism from $L^1(K)$ into $L^1_\tau(G_\tau)$.
\begin{theorem}\label{30}
Let $\tau:H\to Aut(K)$ be a continuous homomorphism and $G_\tau=H\ltimes_\tau K$ also let $\Phi\in L^+_1$.
Then, the linear map $\lambda_\Phi:L^1(K)\to L^1(G_\tau)$ defined by $\psi\mapsto \lambda_\Phi(\psi):=\Phi(\psi)$ is an isometric $*$-homomorphism.
\end{theorem}
\begin{proof}
According to the preceding result and also (\ref{15}) if $\Phi\in L^+_1$ then for each $\psi\in L^1(K)$ we have
$$\|\Phi(\psi)\|_{L^1(G_\tau)}=\|\Phi\|_{L^1(G_\tau)}\|\psi\|_{L^1(K)}
=\|\psi\|_{L^1(K)}.$$
Let $\Phi\in L^+_1$ and also $\psi,\phi\in L^1(K)$. Then for a.e. $(h,k)\in G_\tau$ we have
\begin{equation}\label{30.1}
\Phi(\psi)\st\Phi(\phi)(h,k)=\psi\ast\phi(k)\|\Phi_h\|_{L^1(K)}.
\end{equation}
Because
\begin{align*}
\Phi(\psi)\rst\Phi(\phi)(h,k)
&=\int_H\Phi(\psi)_h\ast\Phi(\phi)_t(k)\delta(t)dt
\\&=\int_H\left(\int_K\Phi(\psi)(h,s)\Phi(\phi)(t,s^{-1}k)ds\right)\delta(t)dt
\\&=\int_H\left(\int_K\psi(s)\phi(s^{-1}k)\left(\int_K\Phi(h,s')ds'\right)\left(\int_K\Phi(t,s'')ds''\right)ds\right)\delta(t)dt
\\&=\int_H\left(\int_K\Phi(h,s')ds'\right)\left(\int_K\Phi(t,s'')ds''\right)\left(\int_K\psi(s)\phi(s^{-1}k)ds\right)\delta(t)dt
\\&=\psi\ast\phi(k)\|\Phi_h\|_{L^1(K)}\int_H\left(\int_K\Phi(t,s'')ds''\right)\delta(t)dt=\psi\ast\phi(k)\|\Phi_h\|_{L^1(K)}.
\end{align*}
Similarly for a.e. $(h,k)\in G_\tau$ we have $\Phi(\psi)\lst\Phi(\phi)(h,k)=\psi\ast\phi(k)\|\Phi_h\|_{L^1(K)}$ which implies (\ref{30.1}).
Now our pervious result (\ref{30.1}) implies that for each $\psi,\phi\in L^1(K)$ and also for a.e. $(h,k)\in G_\tau$ we have
$$\Phi(\psi\ast\phi)(h,k)
=\psi\ast \phi(k)\int_K\Phi(h,s)ds
=\Phi(\psi)\st\Phi(\phi)(h,k).$$
Also for each $\psi\in L^1(K)$ and for a.e. $(h,k)\in G_\tau$ we have
\begin{align*}
\Phi(\psi^*)(h,k)&=\psi^*(k)\int_K\Phi(h,s)ds
\\&=\left(\Phi(\psi)_h\right)^*(k)=\Phi(\psi)^{*_\tau}(h,k).
\end{align*}
\end{proof}
According to Theorem \ref{30}, for each $\Phi\in L^+_1$ the set $\lambda_\Phi(L^1(K))=\{\Phi(\psi):\psi\in L^1(K)\}$ is a closed $*$-subalgebra of $L^1_{\tau_l}(G_\tau)$ and also $L^1_{\tau_r}(G_\tau)$.
The following corollary  is a consequence of the theorem \ref{30}.
\begin{corollary}
{\it Let $\tau:H\to Aut(K)$ be a continuous homomorphism and $G_\tau=H\ltimes_\tau K$ and also let $\Phi\in L_1^+$. If $\mathcal{J}$ is a closed left(right) ideal of $L^1(K)$ then $\lambda_\Phi(\mathcal{J})$ is a closed left(right) $\tau_r$-ideal($\tau_l$-ideal) of $\lambda_\Phi(L^1(K))$.}
\end{corollary}
Next theorem shows that $f\mapsto\widetilde{f}$ is a norm decreasing $*$-homomorphism from $L^1_\tau(G_\tau)$ onto $L^1(K)$.
\begin{theorem}\label{31}
The map defined by $f\mapsto S_K^\tau(f)=\widetilde{f}$ is a norm decreasing $*$-homomorphism from $L^1_\tau(G_\tau)$ onto $L^1(K)$.
\end{theorem}
\begin{proof}
Let $f,g\in L^1_\tau(G_\tau)$. Using Fubini's theorem for a.e. $k$ in $K$ we have
\begin{align*}
\widetilde{f\rst g}(k)
&=\int_{H}f\ast_\tau g(h,k)\delta(h)dh
\\&=\int_H\left(\int_Hf_h\ast g_t(k)\delta(t)dt\right)\delta(h)dh
\\&=\int_H\int_H\left(\int_Kf(h,s)g(t,s^{-1}k)ds\right)\delta(h)\delta(t)dtdh
\\&=\int_K\left(\int_Hf(h,s)\delta(h)dh\right)\left(\int_Hg(t,s^{-1}k)\delta(t)dt\right)ds
\\&=\int_K\widetilde{f}(s)\widetilde{g}(s^{-1}k)ds
=\widetilde{f}\ast\widetilde{g}(k).
\end{align*}
Similarly for a.e. $k\in K$ we have $\widetilde{f\lst g}(k)=\widetilde{f}\ast\widetilde{g}(k).$ Thus we have
$$\widetilde{f\st g}=2^{-1}\left(\widetilde{f\rst g}+\widetilde{f\lst g}\right)=\widetilde{f}\ast\widetilde{g}.$$
Also for a.e. $k$ in $K$ we have
\begin{align*}
\widetilde{f^{*_\tau}}(k)
&=\int_Hf^{*_\tau}(t,k)\delta(t)dt
\\&=\int_Hf_t^*(k)\delta(t)dt
\\&=\Delta_K(k^{-1})\int_H\overline{f_t(k^{-1})}\delta(t)dt={\left(\widetilde{f}\right)}^{*}(k).
\end{align*}
Note that since $\|\widetilde{f}\|_{L^1(K)}\le \|f\|_{L^1(G_\tau)}$ so the map $f\mapsto \widetilde{f}$ is norm-decreasing. Now let $\Phi\in L^+_1$ and $\psi$ in $L^1(K)$ be arbitrary. Then for a.e. $k\in K$ we have
\begin{align*}
\widetilde{\Phi(\psi)}(k)&=\int_H\Phi(\psi)(t,k)\delta(t)dt
\\&=\int_H\psi(k)\left(\int_K\Phi(t,s)ds\right)\delta(t)dt
\\&=\psi(k)\int_H\int_K\Phi(t,s)ds\delta(t)dt=\psi(k).
\end{align*}
Which implies that $\widetilde{L^1_\tau(G_\tau)}=\{\widetilde{f}:f\in L^1_\tau(G_\tau)\}=L^1(K)$.
\end{proof}
Also we can conclude the following corollary. Note that the same result holds for $L^1_{\tau_l}(G_\tau)$.
\begin{corollary}\label{}
{\it The map $f\mapsto S_K^\tau(f)=\widetilde{f}$ is a norm decreasing $*$-homomorphism from $L^1_{\tau_r}(G_\tau)$ onto $L^1(K)$.}
\end{corollary}
Let $\mathcal{J}_\tau^1:=\{f\in L^1_\tau(G_\tau):S_K^\tau(f)=\widetilde{f}=0\}$. Theorem \ref{31} implies that $\mathcal{J}_\tau^1$ is a closed two sided $\tau$-ideal in $L^1_\tau(G_\tau)$.

As an immediate application of the theorem \ref{31} we show that if $\mathcal{I}$ is a closed left(right) $\tau$-ideal of $L^1_\tau(G_\tau)$ with $\mathcal{J}_\tau^1\subseteq\mathcal{I}$ then $\widetilde{\mathcal{I}}=\{\widetilde{f}:f\in \mathcal{I}\}$ is a closed left(right) ideal of $L^1(K)$.
\begin{corollary}\label{32}
{\it Let $\tau:H\to Aut(K)$ be a continuous homomorphism and $G_\tau=H\ltimes_\tau K$. If $\mathcal{I}$ is a closed left(right) $\tau$-ideal of $L^1_\tau(G_\tau)$ with $\mathcal{J}_\tau^1\subseteq\mathcal{I}$ then $\widetilde{I}$ is a closed left(right) ideal of $L^1(K)$.}
\end{corollary}
\begin{proof}
In general setting using Proposition \ref{31} if $\mathcal{I}$ is left(right) $\tau$-ideal of $L^1_\tau(G_\tau)$ then $\widetilde{\mathcal{I}}$ is a left(right) ideal of $L^1(K)$. Now let $\mathcal{I}$ be closed in $L^1_\tau(G_\tau)$. We show that $\widetilde{\mathcal{I}}$ is closed in $L^1(K)$. Let $\psi\in L^1(K)$ and $\{f_n\}\subset \mathcal{I}$ with $\psi=\|.\|_{L^1(K)}-\lim_{n}\widetilde{f_n}$. If $\Phi\in L_1^+$ we have $\Phi(\psi)=\|.\|_{L^1_\tau(G_\tau)}-\lim_n\Phi(\widetilde{f_n})$. Since $\mathcal{J}_\tau^1\subseteq\mathcal{I}$ for each $f\in\mathcal{I}$ we have $\Phi(\widetilde{f})\in\mathcal{I}$. Using proposition \ref{31} and also since for each $n$ we have $\Phi(\widetilde{f_n})\in \mathcal{I}$ and $\mathcal{I}$ is closed we get
$$\Phi(\psi)=\|.\|_{L^1(G_\tau)}-\lim_n\Phi(\widetilde{f_n})\in\mathcal{I}.$$
So we have $\psi=\widetilde{\Phi(\psi)}\in\widetilde{\mathcal{I}}$.
\end{proof}
In the following corollary we show that if $L^1_\tau(G_\tau)$ has an identity then $K$ should be discrete.
\begin{proposition}\label{}
{\it Let $\tau:H\to Aut(K)$ be a continuous homomorphism and $G_\tau=H\ltimes_\tau K$. If Banach algebra $L^1_{\tau_r}(G_\tau)$ or $L^1_{\tau_l}(G_\tau)$ has identity, then $K$ is discrete.}
\end{proposition}
\begin{proof}
It is enough to prove the result for the left $\tau$-convolution. Let $\epsilon_{\tau_l}$ be an identity for the left $\tau$-convolution. Using Theorem \ref{31} for each $f\in L^1_{\tau_l}(G_\tau)$ and also for a.e. $k\in K$ we have
\begin{align*}
\widetilde{f}\ast\widetilde{e_\tau}(k)
&=\widetilde{f\ast(e_\tau)}(k)
\\&=\int_Hf\st e_\tau(h,k)\delta(h)dh
\\&=\int_Hf(h,k)\delta(h)dh=\widetilde{f}(k).
\end{align*}
Now since the linear map $S_K^\tau:L^1_\tau(G_\tau)\to L^1(K)$ is surjective, $\widetilde{e_\tau}$ is an identity for $L^1(K)$. Thus Theorem 19.19 and also Theorem 19.20 of \cite{HR1}  imply that $K$ is discrete.
\end{proof}
\begin{corollary}\label{33}
{\it Let $\tau:H\to Aut(K)$ be a continuous homomorphism and $G_\tau=H\ltimes_\tau K$. If $\tau$-convolution has identity, then $K$ is discrete.}
\end{corollary}
In the next proposition we show that the map $S_K^\tau:L^1_\tau(G_\tau)\to L^1(K)$ given by $S_K^\tau(f)=\widetilde{f} $ is not injective in the general settings. More precisely we prove that $\mathcal{J}_\tau^1=\{0\}$ if and only if $H$ be the trivial group.
\begin{proposition}\label{33.1}
{\it Let $\tau:H\to Aut(K)$ be a continuous homomorphism and $G_\tau=H\ltimes_\tau K$. The linear map $S_K^\tau:L^1_\tau(G_\tau)\to L^1(K)$ is injective if and only if $H$ is trivial group.}
\end{proposition}
\begin{proof}
If $H$ be the trivial group then clearly $\widetilde{f}=0$ implies $f=0$ and so the linear map $S_K^\tau$ is injective.
To prove the converse note that for each nonnegative function $\varphi\in L^1(H)$ with $\|\varphi\|_{L^1(H)}=1$ and also for each $\psi\in L^1(K)$ we can define ${\psi}_\varphi(h,k)=\delta(h^{-1})\varphi(h)\psi(k)$ for a.e $(h,k)\in G_\tau$. Then ${\psi}_\varphi$ belongs to $L^1_\tau(G_\tau)$ and also for a.e. $k\in K$ we have
$\widetilde{{\psi}_\varphi}(k)=\psi(k)$.
Now if $S_K^\tau$ is injective and also $H$ is not the trivial group then for a fixed $\psi\in L^1(K)$ we have
$$\{0\}\subset\{\psi_\varphi-\psi_{\varphi'}:\varphi,\varphi'\in \mathcal{C}_c^+(H), \|\varphi\|_{L^1(H)}=\|\varphi'\|_{L^1(H)}=1,\varphi\not=\varphi'\}\subseteq \mathcal{J}_\tau^1,$$
which contradicts injectivity of the linear map $S_K^\tau$.
\end{proof}
As an applications of Proposition \ref{33.1} we can prove that $L^1_{\tau}(G_\tau)$ is an associative Banach $*$-algebra if and only if $H$ is trivial group.

\begin{corollary}
{\it Let $K$ be seconde countable and $\tau:H\to Aut(K)$ be a continuous homomorphism and also let $G_\tau=H\ltimes_\tau K$. The $\tau$-convolution defined in (\ref{11}) is associative if and only if $H$ is the trivial group.}
\end{corollary}
\begin{proof}
Obviously when $H$ is the trivial group, the $\tau$-convolution is associative. Conversely suppose the $\tau$-convolution is associative and $\widetilde{f}=0$. Let $\{\psi_n\}_{n=1}^\infty$ be a sequence approximate identity for $L^1(K)$ according to Proposition 2.42 of \cite{FollH} and $\Phi\in L^+_1$. Using associativity of the $\tau$-convolution for a.e. $(h,k)\in G_\tau$ we have
\begin{align*}
f(h,k)&=\lim_nf_h\ast\psi_n(k)
\\&=\lim_n\lim_mf_h\ast(\psi_n\ast\psi_m)(k)
\\&=\lim_n\lim_mf_h\ast\left(\widetilde{\Phi(\psi_n)}\ast\widetilde{\Phi(\psi_m)}\right)(k)
\\&=\lim_n\lim_m f\rst\Phi(\psi_n)\rst\Phi(\psi_m)(h,k)
\\&=\lim_n\lim_m f\lst\Phi(\psi_n)\lst\Phi(\psi_m)(h,k)
\\&=\lim_n\lim_m \widetilde{f}\ast\widetilde{\Phi(\psi_n)}\ast(\psi_m)_h(k)=0.
\end{align*}
Now proposition \ref{33.1} guarantee that $H$ is the trivial group.
\end{proof}


Next corollary shows that when $K$ is second countable, the right $\tau$-convolution is commutative if and only if $K$ is abelian and $H$ is the trivial group.


\begin{corollary}\label{33.2}
{\it Let $K$ be seconde countable and $\tau:H\to Aut(K)$ be a continuous homomorphism and also let $G_\tau=H\ltimes_\tau K$. The right $\tau$-convolution  is commutative if and only if $K$ is abelian and $H$ is the trivial group.}
\end{corollary}
\begin{proof}
Clearly when $H$ is the trivial group the right $\tau$-convolution coincides with the standard convolution on $L^1(K)$ and so if $K$ is abelian the left $\tau$-convolution is commutative. Now let right $\tau$-convolution be commutative and also $f,g\in L^1_{\tau_r}(G_\tau)$. Using Proposition \ref{31} for a.e. $k\in K$ we have
\begin{align*}
\widetilde{f}\ast\widetilde{g}(k)=\widetilde{f\rst g}(k)=\widetilde{g\rst f}(k)=\widetilde{g}\ast\widetilde{f}(k).
\end{align*}
Now since $\{\widetilde{f}:f\in L^1_{\tau_r}(G_\tau)\}=L^1(K)$ we get that $L^1(K)$ is commutative and so $K$ is abelian. Also since the right $\tau$-convolution is commutative, using (\ref{13}) for each $f,g\in L^1_{\tau_r}(G_\tau)$ and a.e. $(h,k)\in G_\tau$ we have $f_h\ast\widetilde{g}(k)=g_h\ast \widetilde{f}(k)$. To show that $H$ is the trivial group, assume that $\{\psi_n\}_{n=1}^\infty$ is an approximate identity for $L^1(K)$ and also $\widetilde{f}=0$. Then if $\Phi\in L^+_1$ for a.e. $(h,k)\in G_\tau$ we have
\begin{align*}
f(h,k)&=\lim_nf_h\ast\psi_n(k)
\\&=\lim_nf_h\ast\widetilde{\Phi(\psi_n)}(k)
\\&=\lim_n\Phi(\psi_n)_h\ast \widetilde{f}(k)=0.
\end{align*}
\end{proof}
The same result for the left $\tau$-convolution can be obtained by the similar argument.
In the next theorem we show that $\tau$-convolution is commutative whenever $K$ is abelian and also we prove that the converse is true.
\begin{theorem}\label{34}
{\it Let $\tau:H\to Aut(K)$ be a continuous homomorphism and $G_\tau=H\ltimes_\tau K$. The $\tau$-convolution is commutative if and only if $K$ is abelian.}
\end{theorem}
\begin{proof}
When the $\tau$-convolution is commutative, similar method as we use in Corollary \ref{33.2} works and implies that $K$ is abelian.
Now let $K$ be an abelian group. Thus the Banach $*$-algebra $L^1(K)$ is commutative and so according to the definition of the $\tau$-convolution, for each $f,g\in L^1_\tau(G_\tau)$ and also for a.e. $(h,k)\in G_\tau$ we have
\begin{align*}
f\st g(h,k)
&=2^{-1}\int_H(f_h\ast g_t(k)+f_t\ast g_h(k))\delta(t)dt
\\&=2^{-1}\int_H(g_t\ast f_h(k)+g_h\ast f_t(k))\delta(t)dt=g\st f(h,k).
\end{align*}
\end{proof}
As an immediate consequence of Theorem \ref{34} we show that when $K$ is abelian the $\tau$-convolution and $\tau$-involution makes $L^1(G_\tau)$ into a Jordan Banach $*$-algebra.
\begin{corollary}
{\it Let $\tau:H\to Aut(K)$ be a continuous homomorphism and $G_\tau=H\ltimes_\tau K$ with $K$ abelian. $L^1(G_\tau)$ with respect to the $\tau$-convolution and $\tau$-involution is a Jordan Banach $*$-algebra.}
\end{corollary}
\begin{proof}
If $K$ is abelian, Theorem \ref{34} guarantee that $\tau$-convolution is commutative. According to Theorem \ref{11.3} it is sufficient to show that the $\tau$-convolution satisfies the Jordan identity. Now let $f,g\in L^1(G_\tau)$. Using Proposition \ref{11.2} and also commutativity and associativity of the standard convolution on $L^1(K)$ for a.e $(h,k)\in G_\tau$ we have
\begin{align*}
(f\st g)\st (f\st f)(h,k)-f\st\left(g\st (f\st f)\right)(h,k)
&=(f\rst g)\rst (f\rst f)(h,k)-f\lst\left(g\lst (f\lst f)\right)(h,k)
\\&=(f\rst g)_h\ast\widetilde{(f\rst f)}(k)-\widetilde{f}\ast\left(g\lst (f\lst f)\right)_h(k)
\\&=(f\rst g)_h\ast\left(\widetilde{f}\ast\widetilde{f}\right)(k)-\widetilde{f}\ast\left(\widetilde{g}\ast(f\lst f)_h\right)(k)
\\&=(f_h\ast\widetilde{ g})\ast\left(\widetilde{f}\ast\widetilde{f}\right)(k)-\widetilde{f}\ast\left(\widetilde{g}\ast(\widetilde{f}\ast f_h)\right)(k)=0.
\end{align*}
\end{proof}
\begin{remark}\label{35}
The standard convolution on $L^1(G_\tau)$ is commutative if and only if $G_\tau$ is abelian. Theorem \ref{34} guarantee that the $\tau$-convolution has some differences with standard convolution which can be defined on the function algebra $L^1(G)$ for each arbitrary locally compact group $G$.  But in the sequel we show
that $\tau$-convolution on $L^1(G_\tau)$ coincides with the standard convolution of $L^1(G_\tau)$ only when $H$ is the trivial group which implies that the $\tau$-convolution on $L^1(G_\tau)$ is completely different form the standard convolution of $L^1(G_\tau)$.
\end{remark}
Let $\mathcal{A}$ be a closed $*$-subalgebra of $L^1_\tau(G_\tau)$. A sequence $\{u_n\}_{n=1}^\infty$ in $\mathcal{A}$ is called a $\tau$-sequence approximate identity for the closed $*$-subalgebra $\mathcal{A}$ if for each $f\in \mathcal{A}$ we have
\begin{equation}\label{}
\lim_{n}\|u_n\st f-f\|_{L^1(G_\tau)}=\lim_{n}\|f\st u_n -f\|_{L^1(G_\tau)}=0.
\end{equation}
Next theorem shows us that the non-associative Banach $*$-algebra $L^1_\tau(G_\tau)$ has a $\tau$-sequence approximate identity if and only if $H$ is the trivial group.
\begin{theorem}\label{36}
Let $\tau:H\to Aut(K)$ be a continuous homomorphism and $G_\tau=H\ltimes_\tau K$. The $\tau$-convolution admits a $\tau$-sequence approximate identity if and only if $H$ be the trivial group.
\end{theorem}
\begin{proof}
If $H$ is the trivial group $\{e\}$ then for each continuous homomorphism $\tau:H\to Aut(K)$ we have
$$G_\tau=\{e\}\ltimes_\tau K=K.$$
which implies $L^1(G_\tau)=L^1(K)$. Thus the $\tau$-convolution coincides with the standard convolution of $L^1(K)$ and so any standard approximate identity for $L^1(K)$ is a $\tau$-sequence approximate identity for $L^1(G_\tau)$. Conversely assume that the $\tau$-convolution admits a $\tau$-sequence approximate identity $\{u_n\}_{n=1}^\infty$. Let $f\in\mathcal{J}^1_\tau$, we show that $f(h,k)=0$ for a.e. $(h,k)\in G_\tau$.
By Theorem \ref{31}, $\{\widetilde{u_n}\}_{n=1}^\infty$ is a sequence of approximate identity for $L^1(K)$. Using (\ref{13}) for a.e. $(h,k)\in G_\tau$ we have
\begin{align*}
f(h,k)&=\lim_{n}f\st u_n(h,k)
\\&=2^{-1}\lim_{n}\left(f_h\ast\widetilde{u_n}(k)+\widetilde{f}\ast (u_n)_h(k)\right)
\\&=2^{-1}\lim_{n}f_h\ast\widetilde{u_n}(k)=2^{-1}f(h,k).
\end{align*}
Thus we get $f(h,k)=0$ for a.e. $(h,k)\in G_\tau$ and so we have $\mathcal{J}^1_\tau=\{0\}$.
Now Proposition \ref{33.1} works and implies that $H$ is the trivial group.
\end{proof}
As an immediate consequence of Theorem \ref{36} we achieve that when $H$ and $K$ are second countable locally compact groups the $\tau$-convolution defined coincides with the standard convolution of $L^1(G_\tau)$ if and only if $H$ is the trivial group. Which complectly guarantee that our extension is worthwhile and also not trivial.
\begin{corollary}\label{37}
{\it Let $H$ and $K$ be second countable locally compact groups and $\tau:H\to Aut(K)$ be a continuous homomorphism and also let $G_\tau=H\ltimes_\tau K$. The $\tau$-convolution defined in (\ref{11}) coincides with the standard convolution of $L^1(G_\tau)$ if and only if $H$ is the trivial group.}
\end{corollary}
\begin{proof}
Clearly if $H$ is the trivial group the $\tau$-convolution coincides with the standard convolution of $L^1(G_\tau)$. Conversely if the $\tau$-convolution coincides with the standard convolution of $L^1(G_\tau)$ for each $f,g\in L^1(G_\tau)$ we have $f\st g=f\ast g$. Since the underlying topological space of $G_\tau$ is second countable, the standard convolution on $L^1(G_\tau)$ possess a sequence approximate identity thus the $\tau$-convolution admits a $\tau$-sequence approximate identity and using Theorem \ref{36} the result holds.
\end{proof}
Let $\mathcal{A}$ be a closed $*$-subalgebra of $L^1_\tau(G_\tau)$. A family $\{u_\alpha\}_{\alpha\in I}$ in $\mathcal{A}$ is called a $\tau$-approximate identity for the $*$-subalgebra $\mathcal{A}$ if for each $f\in \mathcal{A}$ we have
\begin{equation}\label{}
\lim_{\alpha\in I}\|u_\alpha\st f-f\|_{L^1_\tau(G_\tau)}=\lim_{\alpha\in I}\|f\st u_\alpha -f\|_{L^1_\tau(G_\tau)}=0.
\end{equation}

Although according to the Theorem \ref{36}, when $H$ and $K$ are second countable if the $\tau$-convolution on $L^1_\tau(G_\tau)$ admits a $\tau$-approximate identity then automatically $H$ should be the trivial group but in the following theorem we show that in general settings if $\Phi\in L^+_1$ the closed
$*$-subalgebra $\lambda_\Phi\left(L^1(K)\right)$ possesses a special kind of $\tau$-approximate identity.
\begin{theorem}\label{40}
{\it Let $\tau:H\to Aut(K)$ be a continuous homomorphism and $G_\tau=H\ltimes_\tau K$ and also let $\Phi\in L^+_1$.
The closed $*$-subalgebra $\lambda_\Phi\left(L^1(K)\right)$ admits a bounded $\tau$-approximate identity.}
\end{theorem}
\begin{proof}
Let $\{\psi_\alpha\}_{\alpha\in I}$ be a bounded approximate identity for $L^1(K)$ with $\|\psi_\alpha\|_{L^1(K)}=1$, $\psi_\alpha=\psi_\alpha^*$ and $\psi_\alpha\ge0$ for each $\alpha\in I$. Thus for each $\phi\in L^1(K)$ we have
\begin{equation}\label{40.1}
\lim_{\alpha\in I}\|\psi_\alpha\ast\phi-\phi\|_{L^1(K)}=\lim_{\alpha\in I}\|\phi\ast\psi_\alpha-\phi\|_{L^1(K)}=0.
\end{equation}
Now let $\Phi\in L^+_1$ and also for each $\alpha\in I$ let $u_\alpha(h,k):=\Phi(\psi_\alpha)(h,k)$. Then for each $\alpha\in I$ we have $u_\alpha\in L^1_\tau(G_\tau)$ with $\|u_\alpha\|_{L^1_\tau(G_\tau)}=1$ and also $u_\alpha=u_\alpha^{*_\tau}$. We show that $\{u_\alpha\}_{\alpha\in I}$ is a $\tau$-approximate identity for $\Phi\left(L^1(K)\right)$. Using Theorem \ref{30} and also (\ref{40.1}) for each $f\in \lambda_\Phi\left(L^1(K)\right)$ with $f=\Phi(\phi)$ we have
\begin{align*}
\lim_{\alpha\in I}\|u_\alpha\st f-f\|_{L^1_\tau(G_\tau)}
&=\lim_{\alpha\in I}\|\Phi(\psi_\alpha)\st\Phi(\phi)-\Phi(\psi)\|_{L^1_\tau(G_\tau)}
\\&=\lim_{\alpha\in I}\|\Phi\left(\psi_\alpha\ast\phi-\psi\right)\|_{L^1_\tau(G_\tau)}
\\&=\lim_{\alpha\in I}\|\psi_\alpha\ast\phi-\psi\|_{L^1(K)}=0.
\end{align*}
Similarly we have
$$\lim_{\alpha\in I}\|f\st u_\alpha-f\|_{L^1_\tau(G_\tau)}=0.$$
\end{proof}

\section{{\bf $L^p(G_\tau)$ as a left Banach $L^1_{\tau_l}(G_\tau)$-module}}

The left $\tau$-convolution defined in (\ref{10.l}) can be extended from $L^1(G_\tau)$ to other $L^p(G_\tau)$ spaces with $1\le p\le\infty$. In this section we make $L^p(G_\tau)$ into a left Banach $L^1(G_\tau)$-module. First we define a module action.

Let the left module action $\lst_{(p)}:L^1_{\tau_l}(G_\tau)\times L^p(G_\tau)\to L^p(G_\tau)$ defined via $(f,u)\mapsto f\lst_{(p)}u$ where $f\lst_{(p)}u$ is given by
\begin{equation}\label{51}
f\lst_{(p)}u(h,k):=\int_Hf_t\ast u_h(k)\delta(t)dt.
\end{equation}
The left module action defined in (\ref{51}) for a.e. $(h,k)\in G_\tau$ can be written in the form
\begin{equation}\label{52}
f\st_{(p)}u(h,k)=\widetilde{f}\ast u_h(k).
\end{equation}
Because using Fubini's theorem for a.e. $(h,k)\in G_\tau$ we have
\begin{align*}
f\st_{(p)}u(h,k)
&=\int_Hf_t\ast u_h(k)\delta(t)dt
\\&=\int_H\left(\int_Kf(t,s)u(h,s^{-1}k)ds\right)\delta(h)dh
\\&=\int_K\left(\int_Hf(t,s)\delta(t)dt\right)u(h,s^{-1})kds=\widetilde{f}\ast u_h(k).
\end{align*}
The module action defined in (\ref{51}) is converges and also belongs to $L^p(G_\tau)$. Because using Fubini's theorem and also
Proposition 2.39 of \cite{FollH} we have
\begin{align*}
\|f\lst_{(p)}u\|_{L^p(G_\tau)}^p
&=\int_K\int_H|f\lst_{(p)}u(h,k)|^p\delta(h)dhdk
\\&=\int_K\int_H|\widetilde{f}\ast u_h(k)|^p\delta(h)dhdk
\\&=\int_H\left(\int_K|\widetilde{f}\ast u_h(k)|^pdk\right)\delta(h)dh
\\&=\int_H\|\widetilde{f}\ast u_h\|_{L^1(K)}^p\delta(h)dh
\\&\le\|\widetilde{f}\|_{L^1(K)}^p\int_H\|u_h\|_{L^p(K)}^p\delta(h)dh
=\|f\|_{L^1_{\tau_l}(G_\tau)}^p\|u\|_{L^p(G_\tau)}^p.
\end{align*}

\begin{remark}\label{}
When $p=1$ the left module action defined in (\ref{51}) coincides with the left $\tau$-convolution defined in (\ref{10.l}). Also when $H$ is the trivial group the left module action defined in (\ref{51}) coincides with the left module action defined on $L^p(K)$ via Proposition 2.39 of \cite{FollH}.
\end{remark}


In the next theorem we show that the left module action defined in (\ref{51}) is associative and so makes $L^p(G_\tau)$ into a left Banach $L^1_{\tau_l}(G_\tau)$-module.


\begin{theorem}\label{53}
Let $\tau:H\to Aut(K)$ be a continuous homomorphism and $G_\tau=H\ltimes_\tau K$ and also let $1\le p\le\infty$.  $L^p(G_\tau)$ with respect to the left module action $\st_{(p)}$ defined in (\ref{51}) is a left Banach $L^1_{\tau_l}(G_\tau)$-module.
\end{theorem}
\begin{proof}
It is sufficient to show that the left module action (\ref{51}) is associative. Let $f,g\in L^1_{\tau_l}(G_\tau)$ and $u\in L^p(G_\tau)$. Using Theorem \ref{31} and also (\ref{52}) for a.e. $(h,k)\in G_\tau$ we have
\begin{align*}
\left(f\lst g\right)\st_{(p)}u(h,k)
&=\widetilde{(f\lst g)}\ast u_h(k)
\\&=(\widetilde{f}\ast\widetilde{g})\ast u_h(k)
\\&=\widetilde{f}\ast(\widetilde{g}\ast u_h)(k)
\\&=\widetilde{f}\ast(g\lst_{(p)} u)_h(k)=f\lst_{(p)}\left(g\lst_{(p)}u\right)(h,k).
\end{align*}
\end{proof}


A sequence $\{f_n\}_{n=1}^\infty$ in $L^1_{\tau_l}(G_\tau)$ is called a left $\tau_l$-sequence approximate identity, if for each $p\ge1$ and $u\in L^p(G_\tau)$ satisfies
\begin{equation}\label{54}
\lim_{n}\|f_n\lst u-u\|_{L^p(G_\tau)}=0.
\end{equation}
In the following theorem we show that when $K$ is a second countable locally compact group the left Banach $L^1_{\tau_l}(G_\tau)$-module $L^p(G_\tau)$ admits a left $\tau_l$-sequence approximate identity.


\begin{theorem}\label{55}
Let $\tau:H\to Aut(K)$ be a continuous homomorphism and $G_\tau=H\ltimes_\tau K$ with $K$ second countable. The Banach algebra $L^1_{\tau_l}(G_\tau)$ possesses a left $\tau_l$-sequence approximate identity.
\end{theorem}
\begin{proof}
Let $\{\psi_n\}_{n=1}^\infty$ be an approximate identity for $L^1(K)$ according to Proposition 2.42 of \cite{FollH}. Thus for each $p\ge1$ and also $\varphi\in L^p(K)$ we have
\begin{equation}\label{55.1}
\lim_{n}\|\psi_n\ast\varphi-\varphi\|_{L^p(G_\tau)}=0.
\end{equation}
Let $\Phi\in L_1^+$ and for all $n$ put $f_n(h,k):=\Phi(\psi_n)(h,k)$ for a.e. $(h,k)\in G_\tau$. We show that $\{f_n\}_{n=1}^\infty$ a left $\tau_l$-sequence approximate identity. Let $p\ge 1$ and also $u\in L^p(G_\tau)$. Using (\ref{51}), (\ref{55.1}) and also Theorem \ref{31} and The Dominated Convergence Theorem
we achieve
\begin{align*}
\lim_{n}\|f_n\lst u-u\|_{L^p(G_\tau)}^p
&=\lim_{n}\int_H\left(\int_K|f_n\lst u(h,k)-u(h,k)|^pdk\right)\delta(h)dh
\\&=\lim_{n}\int_H\left(\int_K|\Phi(\psi_n)\lst u(h,k)-u(h,k)|^pdk\right)\delta(h)dh
\\&=\lim_{n}\int_H\left(\int_K|\widetilde{\Phi(\psi_n)}\ast u_h(k)-u_h(k)|^pdk\right)\delta(h)dh
\\&=\lim_{n}\int_H\left(\int_K|\psi_n\ast u_h(k)-u_h(k)|^pdk\right)\delta(h)dh
\\&=\lim_{n}\int_H\left(\|\psi_n\ast u_h-u_h\|_{L^p(K)}^p\right)\delta(h)dh
\\&=\int_H\left(\lim_n\|\psi_n\ast u_h-u_h\|_{L^p(K)}^p\right)\delta(h)dh=0.
\end{align*}
\end{proof}


As an immediate consequence of Theorem \ref{55} we have the following corollary.


\begin{corollary}\label{}
{\it Let $\tau:H\to Aut(K)$ be a continuous homomorphism and $G_\tau=H\ltimes_\tau K$ with $K$ second countable. The Banach algebra $L^1_{\tau_l}(G_\tau)$ possesses a $\tau_l$-sequence approximate identity.}
\end{corollary}


\bibliographystyle{amsplain}

\end{document}